\documentclass{amsart} 
\usepackage{amsmath,amssymb} 
\usepackage{enumerate}  

\title[Renewal Sanov theorem]{A renewal version of the Sanov theorem} 

\author[M.\ Mariani]{Mauro Mariani} 
\address{Mauro Mariani,
  Dipartimento di Matematica,
  Universit\`a degli Studi di Roma La Sapienza
Piazzale Aldo Moro 5, 00185 Roma} 
\email{mariani@mat.uniroma1.it}
 \author[L.\ Zambotti]{Lorenzo Zambotti}
 \address{LPMA (CNRS UMR. 7599) Universit\'e Paris 6 --
  Pierre et Marie Curie, U.F.R. Math\'ematiques,
 Case 188, 4 place Jussieu, 75252 Paris cedex 05}
\email{lorenzo.zambotti@upmc.fr} 



\newtheorem{theorem}{Theorem}[section]
\newtheorem{proposition}[theorem]{Proposition}
\newtheorem{definition}[theorem]{Definition}
\newtheorem{lemma}[theorem]{Lemma}

\newtheorem{remark}[theorem]{Remark}

\newcommand{\mc}[1]{{\mathcal #1}}
\newcommand{\mf}[1]{{\mathfrak #1}}
\newcommand{\mb}[1]{{\mathbf #1}}
\newcommand{\bb}[1]{{\mathbb #1}}
\newcommand{\eps}{\varepsilon}

\newfont{\indic}{bbmss12}
\newcommand{\un}[1]{\hbox{{\indic 1}$_{#1}$}}

\newcommand{\refa}[1]{(A\ref{a:#1})}

\begin{document}
\begin{abstract}
Large deviations for the local time of a process $X_t$ are
  investigated, where $X_t=x_i$ for $t \in [S_{i-1},S_i[$ and $(x_j)$
  are i.i.d.\ random variables on a Polish space, $S_j$ is the $j$-th
  arrival time of a renewal process depending on $(x_j)$. No moment conditions are assumed on the arrival times of the renewal process.
\end{abstract}

\maketitle

\section{Main results}
\label{s:1}
\subsection{Outline of the result}
\label{ss:1.1}
 Consider an i.i.d.\ sequence $(x_i)_{i\in \bb N^+}$ in a Polish space $\mc X$, with marginal distribution $\bar \mu
$. One may define a stochastic process 
$(X_t)_{t\geq 0}$ 
on $\mc X$ by setting $X_t=x_i$ for $t\in [i-1,i[$, and consider its empirical measure $\pi_t:=\tfrac 1t\int_{[0,t[}ds\,
\delta_{X_s}$. The ergodic theorem 
then states that 
$\pi_t \to \bar \mu$ as $t\to +\infty$, while the Sanov theorem yields a finer estimate for the probability that $\pi_t$ 
is found in a small neighborhood of a given Borel probability measure $\bar \nu$ on $\mc X$.  Such probability is estimated, in the sense of large deviations, as $\exp(-t H(\bar \nu | \bar \mu))$, where $H(\bar \nu | \bar \mu)$ is the relative entropy of $\bar \nu$ with respect to $\bar \mu$. 

In this paper, we want to provide a similar result, in the case in which the time spent by the process $X_t$ at the point $x_i$ may depend on the process itself. In particular, for $\tau \colon \mc X \to [0,+\infty]$ a measurable map, define $\mc N_t:= \inf \{n \in \bb N^+\,: \: \sum_{i=1}^{n+1} \tau(x_i)\ge t\}$, and $X_t:= x_{\mc N_t+1}$. In the 
next section, the precise mathematical setting for the study of the large deviations of the empirical measure of $X_t$ is recalled, and a large deviations result is established in Section~\ref{ss:1.4}. While for $\tau \equiv 1$ one gets the classical Sanov theorem, we are mainly interested in the case where the law of $\tau$ under $\bar \mu$ features heavy tails. In such a case the Markov process $(X_t,t-\sum_{i=1}^{\mc N_t} \tau(x_i))$ does not have good ergodic properties, and the classical Donsker-Varadhan theorem is violated.

\subsection{Mathematical setting}
\label{ss:1.2}
In the following $\bb N=\{0,\,1, \ldots \}$, $\bb N^+=\bb N\setminus \{0\}$; $\mc X$ is a Polish space, that is a separable, completely metrisable topological space; a general element of $\mc X^{\bb N^+}$ will be denoted $\mb x=(x_1,\,x_2,\ldots)$; $C_b(\mc X)$ and $C_c(\mc X)$ are respectively the spaces of bounded continuous functions and compactly supported continuous functions on $\mc X$. $\mc M_1(\mc X)$ is the space positive Radon measure on $\mc X$ with total variation bounded by $1$, while $\mc P(\mc X) \subset \mc M_1(\mc X)$ is the set of Borel probability measures on $\mc X$. For $\mu \in \mc M_1(\mc X)$ and $f$ a $\mu$-integrable function, we write $\mu(f):=\int d\mu\,f$. For $\mu,\,\nu \in \mc P(\mc X)$, $\mb H(\nu|\mu) $ denotes the relative entropy of $\nu$ with respect to $\mu$:
\begin{equation}
\label{e:relent1}
\mb H(\nu|\mu):=\sup_{\varphi \in C_b(\mc X)} \nu(\varphi)-\log \mu(e^\varphi)=
\begin{cases}
\int \mu(dx) h\big(\frac{d\nu}{d\mu}\big) & \text{if $\nu <<\mu$;}
\\
+\infty & \text{otherwise;}
\end{cases}
\end{equation}
where the positive convex function $h$ is defined as $h(\varrho)=\varrho \,(\log \varrho -1)+1$.

We always consider $\mc P(X)$ equipped with the narrow (or weak) topology, namely the weakest topology such that $\mu \mapsto \mu(f)$ is continuous for all $f\in C_b(\mc X)$. In the particular case in which $\mc X$ is locally compact, we will also regard $\mc M_1(\mc X)$ as a topological space, equipped with the vague topology, namely the weakest topology such that $\mu \mapsto \mu(f)$ is continuous for all $f\in C_c(\mc X)$. $\mc P(X)$ is then a Polish space, and if $\mc X$ is locally compact $\mc M_1(\mc X)$ is a compact Polish space.

Fix a reference probability $\bar \mu \in \mc P(\mc X)$ and a measurable function
$\tau \colon \mc X \to [0,+\infty]$; $\tau(x)$ has to be interpreted as the time
elapsed at $x$. $\bar \mu$ and $\tau$ are the only 'inputs' of the problem.

Define $\xi \colon \mc X\to [0,+\infty]$ and $\xi^\infty\in [0,+\infty]$ as
\begin{equation}
 \label{e:xi} 
 \begin{split}
 \xi(x)=\inf_{\delta>0} \,
\sup \big\{c\ge 0 \,:\: \bar \mu(e^{c \tau} \un{B_\delta(x)})<+\infty  \big\}
\\
\xi^\infty:= \sup_{K \subset \mc X,\,K \text{compact}} \sup \big\{c\ge 0\,:\: \bar \mu(e^{c \tau} \un{K^c})<+\infty  \big\}
 \end{split}
\end{equation}
where $B_\delta(x) \subset \mc X$ is the ball of radius $\delta$
centered at $x$, see \eqref{e:xi2} for another characterisation of $\xi$. 
Note $\xi^\infty=+\infty$ if $\mc X$ is compact. 

The role of the auxiliary function $\xi$ and of the assumptions below are discussed at the end of this section. In particular it is remarked that \refa{2} below is implied by regularity assumptions on $\tau$ (e.g.\ upper semicontinuity at infinity). Hereafter \refa{1} and \refa{2} will \emph{always} be assumed, while our main results are proved whenever at least one of \refa{3} or \refa{4} holds (with somehow different statements in the two cases).
\begin{enumerate}[{(A}1{)}]
\item \label{a:1} $\bar \mu(\{\tau=0\})=\bar \mu(\{\tau=+\infty\})=0$.
\item \label{a:2} $\bar \mu(\{\xi <+\infty\})=0$.
\item \label{a:3} $\xi^\infty=+\infty$.
\item \label{a:4} $\mc X$ is locally compact.
\end{enumerate}

In the following $\mb x$ is sampled as an i.i.d.\ sequence with marginal law
$\bar \mu$ and $\mb E$ will denote the expectation of functions of
$\mb x$ with respect to $\bar \mu^{\otimes \bb N^+}$. By \refa{1}, for each $n \in \bb N$, $t\ge 0$ and a.e.\ $\mb x$, the following random variables are well defined
\begin{equation*}
\begin{split}
& S_0 \equiv S_0(\mb x) :=  0, \qquad S_{n} \equiv S_n(\mb x)  :=\sum_{i=1}^n \tau(x_i), \qquad n\ge 1,
\end{split}
\end{equation*}
\begin{equation*}
\begin{split}
  & \mc N_t \equiv \mc N_t(\mb x) :=\inf \{n\in \bb N\,:\: S_{n+1} \ge t \} =
  \sum_{n=1}^{+\infty} \un {(S_n \le t)},
\end{split}
\end{equation*}
\begin{equation*}
X_t \equiv X_t(\mb x):= x_{\mc N_t+1},
\end{equation*}
\begin{equation}
\label{e:pi}
\pi_t \equiv \pi_t(\mb x) = \frac 1t \int_{[0,t[}\!ds\, \delta_{X_s} \in \mc P(\mc X).
\end{equation}
In other words, $X_t=x_1$ for $t \in [0, \tau(x_1)[$, $X_t=x_2$ for
$t\in [\tau(x_1), \tau(x_1)+\tau(x_2)[$ and so on, while $\pi_t \colon \mc X^{\bb N^+} \to \mc P(\mc X)$
is the local time or the empirical measure of $X_t$. Let $\mb P_t:= \bar\mu^{\otimes \bb N^+}
\circ \pi_t^{-1}$ be the law of $\pi_t$.

From the ergodic theorem, one expects $\pi_t$ to concentrate on a deterministic limit as $t\to +\infty$ (this is easily established, for instance, whenever $\bar \mu(\tau)<+\infty$). Large deviations of $\mb P_t$ are then relevant, and subject of investigation of this paper.

\subsection{Some examples}
\label{ss:1.3}
Taking advantage of the general metric setting, one is able to fit in this framework also the case of a  process with random waiting time, see the examples \eqref{ii1:b} and \eqref{ii1:c} below.
\begin{enumerate}[\upshape(a)]
\item \label{ii1:a} If $\tau(x)\equiv 1$, then we are in the framework of the classical Sanov theorem, \cite[Chapter~6.2]{DZ}. Here $\xi(x)=\xi^\infty=+\infty$ for all $x\in \mc X$.

\item \label{ii1:b} Assume $\mc X=\mc Y \times [0,+\infty]$ for some Polish space $\mc Y$. Let $p$ be a Borel probability on $\mc Y$ and for $p$-a.e. $y$ let $\phi_y$ be a probability on $[0,+\infty]$ concentrated on $]0,+\infty[$, with $y\mapsto \phi_y$ measurable. Set $d\bar \mu((y,t))= d p(y)\,d\phi_y(t)$ and $\tau(y,t)=t$. Then we are in the framework of a pure jump process, jumping on $\mc Y$ with law $p$ and spending a random time at a visited point $y$ with law $\phi_y$. In this case
\begin{equation*}
\xi(y,t)=
\begin{cases}
\sup \{c\ge 0\,:\int \phi_y(ds) e^{c s} <+\infty\} & \text{if $t=+\infty$ and $y\in \mathrm{Supp}(\nu)$}
\\
+\infty & \text{otherwise.}
\end{cases}
\end{equation*}
\begin{equation*}
\xi^\infty=\sup_{K\subset \mc Y,\, K \text{compact}} \inf_{y\in K^c} \xi(y,+\infty)
\end{equation*}

\item \label{ii1:c} As a special case of \eqref{ii1:b}, take $\mc X:= [0,+\infty[ \times  [0,+\infty]$ and for $\bar\mu(d(r,s))=\nu(dr)\phi(ds)$, where $\nu$ is any probability measure on $]0,+\infty[$ and $\phi$ is the exponential law with mean 1. Set $\tau((y,s))=\theta(y) s$, so that, conditionally on $y$, $\tau$ is an exponential random variable with mean $\theta(y)$. In this setting, $\mc N_t$ is an inhomogeneous Poisson random process, and the empirical measure $\pi_t$ keeps track of the rates of the interarrival times. In this case  $\xi(y,t)=+\infty$ for $t<+\infty$ or $y \not \in \mathrm{Supp}(\nu)$, while $\xi(y,+\infty)=1/\theta(y)$ for $y\in \mathrm{Supp}(\nu)$, and $\xi^\infty=\varliminf_{y\to+\infty} \xi(y,+\infty)$.

\item \label{ii1:d} An interesting example in which $\tau$ is 'truly' deterministic is the following. $\mc X= ]0,+\infty[^n$, $\bar \mu(dx)=\prod_{i=1}^n \bar \mu_i(dx_i)$ for some probabilities $\bar \mu_i \in \mc P([0,+\infty[)$ and $\tau(x)= \tfrac 1n \sum_{i=1}^n \tfrac{1}{x_i}$. This is a model for a particle moving on $1$-dimensional torus of length $1$. During its motion the particle touches some fixed \emph{hot} points equi-spaced on the torus, and it changes its speed by sampling a new one with law $\bar \mu_i$ at the hot point $i$. $\tau(x)$ is then the time elapsed to complete a tour of the torus.

One can derive the large deviations of some physical quantities (e.g.\ kinetic energy of the particle) from the large deviations of the empirical measure of $X_t$. The physically relevant case is $\bar \mu_i(x_i)= x_i e^{-\beta_i x_i^2}dx_i$ for some $\beta_i>0$. Then $\xi^\infty=+\infty$ and $\xi(x)=+\infty$ unless one the $x_i$ is $0$, in which case $\xi(x)=0$. As remarked below, when $\{\xi=0\}$ is non-empty, the large deviations rate functional is not strictly convex. For $n=1$, this moving particle dynamics has been used as a building block of a toy model of out-of-equilibrium statistical mechanics in \cite{LMZ3}, where the absence of strict convexity of the rate causes a dynamic phase transition in the model.
\end{enumerate}

\subsection{Large deviations results}
\label{ss:1.4}
We recall the following standard definition.
\begin{definition}
Let $\mc Y$ be Polish space and $({\mb Q}_t)_{t>0}$ a family of Borel probability measures on
  $\mc Y$ and $I\colon \mc Y\to [0,+\infty]$. Then:
  \begin{itemize}
\item $I$ is \emph{good} if $\{y\in \mc Y\,:\:I(y)\le M\}$ is compact in $\mc Y$ for all $M>0$ and $I\not\equiv+\infty$.  

\item $({\mb Q}_t)_{t>0}$ satisfies a \emph{large deviations upper bound} with good rate $I$ if
\begin{equation*}
\varlimsup_{t \to +\infty} \frac{1}{t} \log {\mb Q}_t(\mc C) \le - \inf_{u \in \mc C} I(u) \qquad \text{for all $\mc C \subset \mc Y$ closed.}
\end{equation*}

\item $({\mb Q}_t)_{t>0}$ satisfies a \emph{large deviations lower bound} with good rate $I$, if
\begin{equation*}
 \varliminf_{t \to +\infty} \frac{1}{t}  \log {\mb Q}_t(\mc O)  \ge - \inf_{u \in \mc O}  I(u) \qquad \text{for all $\mc O\subset \mc Y$ open.}
\end{equation*}
\end{itemize}
$({\mb P}_t)_{t>0}$ is said to satisfy a \emph{good large deviations principle} if both the upper and lower bounds hold with the same good rate $I$.
\end{definition}

For $\nu \in \mc M_1(\mc X)$, let $\nu_a$ and $\nu_s$ be respectively
the absolutely continuous and singular parts of $\nu$ with respect to
$\bar \mu$. If $\nu$ is such that $ \nu(1/\tau) \in ]0,+\infty[$
define $\bar \nu \in \mc P(\mc X)$ as
\begin{equation}
 \label{e:nubar}
 \bar \nu(dx) = \frac{\frac{1}{\tau(x)} \nu(d x) }{ \nu(1/\tau)}.
\end{equation}

\begin{proposition}
 \label{p:Ilsc}
 Define $I \colon \mc P(\mc X)\to [0,+\infty]$ as
\begin{equation*}
I(\nu)=
\begin{cases}
\nu_a(1/\tau) \mb H(\bar{\nu}_a | \bar \mu) + \nu_s(\xi) 
        & \text{if $\nu_a(1/\tau) <+\infty$,}
\\
+\infty & \text{otherwise,}
\end{cases}
\end{equation*}
where we define $\nu_a(1/\tau) \mb H(\bar{\nu}_a|\bar \mu)=0$ whenever $\nu_a(1/\tau)=0$. If \refa{3} holds, then $I$ is a good and convex functional on $\mc P(\mc X)$.
\end{proposition}
\begin{theorem}
 \label{t:ld1}
If \refa{3} holds, then $(\mb P_t)_{t>0}$ satisfies a good large deviations principle on $\mc P(\mc X)$ with rate $I$.
\end{theorem}

In the following remark some features of the functional $I$ are investigated. In particular we characterise the cases where $I$ is strictly convex and those in which it features affine stretches.
\begin{remark}
\label{r:conv}
Assume \refa{3}. Since $\xi(x)=+\infty$ for $x\not \in \mathrm{Supp}(\bar \mu)$,
 $I(\nu)=+\infty$ if $\mathrm{Supp}(\nu) \not \subset \mathrm{Supp}(\bar\mu)$. However, contrary to classical Sanov theorem, in general $I(\nu)<+\infty$ does not imply that $\nu$ is absolutely continuous with respect to $\bar \mu$, unless $\xi \equiv \infty$. In general, the nature of $I(\nu)$ depends on the values of $\xi$ and $\bar\mu(\tau)$. Indeed let 
\begin{equation*}
E:=\{x\in \mc X\,:\:\xi(x)=0\}
\end{equation*}
be the set of points around which $\tau$ has no local exponential moments. Then
 \begin{itemize}
\item[(1)] If $E=\emptyset$, namely if $\xi(x)>0$ for all $x\in \mc X$, then a fortiori $\bar \mu(\tau)<+
\infty$ and $I(\nu)=0$ iff $\nu=\mu$, where (consistently with \eqref{e:nubar})
\begin{equation}
\label{e:mu}
\mu(dx):=
  \frac{\tau(x) \bar \mu(dx)}{\bar \mu(\tau)}.
\end{equation}

\item[(2)] If $E \neq \emptyset$, there are two possibilities
  \begin{itemize}
      \item[(2A)] If $\bar \mu(\tau)<+\infty$, then $I(\nu)=0$ iff $\nu=\alpha \mu +(1-\alpha) \lambda$ for some $
\alpha \in [0,1]$ and some $\lambda \in \mc P(X)$ such that $\lambda(E)=1$, where $\mu$ is given by \eqref{e:mu}.

      \item[(2B)]If $\bar \mu(\tau)=+\infty$ then $I(\nu)=0$ iff $\nu$ is concentrated on $E$.
   \end{itemize}
 \end{itemize}
In particular, Theorem~\ref{t:ld1} implies the convergence in law of $\pi_t$ to $\mu$ in case (1), and in case (2B) 
if $E$ is a singleton. In all other cases, a nontrivial second order large deviations may hold, see \cite{tsi} where moderate deviations are discussed in a particular case. Finally, if $E \neq \emptyset$, then the subdifferential of $I$ is nontrivial.
\end{remark}

If $\xi^\infty<+\infty$, $({\mb P}_t)_{t>0}$ is not exponentially tight on
$\mc P(\mc X)$, and large deviations need to be
investigated on $\mc M_1(\mc X)$. However, in this case we need $\mc X$ to be locally compact in order to have good topological properties of $\mc M_1(\mc X)$.

\begin{proposition}
 \label{p:Iprime}
Define $I' \colon \mc M_1(\mc X)\to [0,+\infty]$ as
\begin{equation*}
I'(\nu)=
\begin{cases}
\nu_a(1/\tau) \mb H(\bar{\nu}_a|\bar \mu) + \nu_s(\xi) 
  + (1-\nu(\mc X))\xi^\infty
& \text{if $\nu_a(1/\tau) <+\infty$,}
\\
+\infty & \text{otherwise,}
\end{cases}
\end{equation*}
If \refa{4} holds, then $I'$ is a good and convex functional on $\mc M_1(\mc X)$.
\end{proposition}
\begin{theorem}
 \label{t:ld2}
 If \refa{4} holds, then $(\mb P_t)_{t>0}$ satisfies a good large deviations principle on $\mc M_1(\mc X)$ with rate $I'$.
\end{theorem}

%
Under \refa{1}, the key assumption \refa{2} is satisfied whenever
\begin{equation*}
\bar \mu
\big(\cap_{M>0}\mathrm{Closure}(\{\tau \ge M\})\big)=0.
\end{equation*}
In particular \refa{2} holds if $\tau$ is upper
semicontinuous at infinity. Since all the results stated above make sense even dropping \refa{2},
one may wonder whether it is a merely technical condition. While one can prove the large deviations upper bound even dropping this assumption, the lower bound is in general false if \refa{2} does not hold.

\subsection{Outlook}
With the same notation as above, one may also introduce the Markov process $Y_t=(X_t, \frac{t-\sum_{i=1}^{\mc 
N_t}\tau(x_i)}{\tau(x_{\mc N_t+1})}) \in \mc X \times [0,1[$. Large deviations for the empirical measure of $Y_t$ would give large deviations of $X_t$ by a standard contraction argument. Moreover, the Donsker-Varadhan theory \cite{DV} and its extensions provide general large deviations results for the empirical measure of a Markov process. However, this approach fails in this case. On the one hand, standard Donsker-Varadhan theorems cannot be applied here, since $Y_t$ only enjoys weak ergodic properties. On the other hand, even formally, the 
Donsker-Varadhan rate functional does not provide the right answer, a feature already remarked in \cite{LMZ1} for renewal processes. Indeed, it has been proved in \cite{LMZ2} that in general the empirical measure of  $Y_t$ \emph{does not satisfy} a large deviations principle, and in the special case it does (which depends on the law of $\tau$ under $\bar \mu$), the rate functional does not correspond to the Donsker-Varadhan functional. Similarly, the large deviations rate functional for $\pi_t$ does not correspond in general to the one predicted by applying contraction to the Donsker-Varadhan functional for the empirical measure of $Y_t$ (unless $\tau$ has all exponential moments bounded). In this respect, it may be remarkable that the law of $\pi_t$ satisfies a large deviations principle at all.

\section{The functional $I$}
\label{s:2}
This section is devoted to prove Proposition~\ref{p:Ilsc}, Proposition~\ref{p:Iprime} and general properties of the functional $I$, which will play a key role in the proof of the main theorems. First we remark that one can reduce to the case of a compact state space $\mc X$.

\begin{proposition}
\label{p:embedding}
Suppose that Proposition~\ref{p:Ilsc} and Theorem~\ref{t:ld1} hold with the additional hypotheses of $\mc X$ being a \emph{compact} Polish space. Then Proposition~\ref{p:Ilsc}, Theorem~\ref{t:ld1}, Proposition~\ref{p:Iprime} and Theorem~\ref{t:ld2} hold.
\end{proposition}
\begin{proof}
An arbitrary Polish space $\mc X$ embeds continuously in the compact Polish space $[0,1]^{\bb N}$, see \cite[Lemma 3.1.2]{S1}. Regard $\mc X$ as a subset of $[0,1]^{\bb N}$ and let $\mc Y$ be the closure of $\mc X$. Then $\mc Y$ is compact. Extend $\bar \mu$ to $\mc Y$ setting $\bar \mu (\mc Y\setminus \mc X)=0$ and extend $\tau$ to $\mc Y$ setting $\tau(x)=+\infty$ for $x\in \mc Y \setminus \mc X$. We denote $\xi_{\mc Y}$ and $I_{\mc Y}$ the object corresponding to $\xi$ and $I$ on $\mc Y$. Then \refa{1}, \refa{2} hold on $\mc Y$ since they hold on $\mc X$, while refa{3} is trivially satisfied on $\mc Y$. Thus, by the hypotheses of this proposition, the extension of $\mb P_t$ to $\mc P(\mc Y)$ satisfies a large deviations principle with good rate $I_{\mc Y}$. We then separate the two cases, wether \refa{3} or \refa{4} hold on $\mc X$.

If \refa{3} holds (on $\mc X$), then $\xi_{\mc Y}(x)=+\infty$ for $x\in \mc Y\setminus X$ (since neighborhoods of such points $x$ in $\mc Y$ are exactly complements of compact subsets of $\mc X$). Thus the map $\Pi \colon \mc P(\mc Y) \to \mc P(\mc X)$ defined as
\begin{equation*}
\Pi(\nu)=
\begin{cases}
\nu(\cdot|\mc X):= \frac{\nu(\cdot\,\cap \mc X)}{\nu(\mc X)} & \qquad \text{if $\nu(\mc X)>0$}
\\
\bar \mu & \qquad \text{otherwise}
\end{cases}
\end{equation*}
is continuous on the domain of $I_{\mc Y}$. Since $\Pi$ is just the restriction map for probabilities concentrated on $\mc X$, the extension of $\mb P_t$ to $\mc P(\mc Y)$ is mapped to $\mb P_t$ by $\Pi$. 
Then by contraction principle \cite[Chapter 4.2]{DZ}, $I$ is good and $\mb P_t$ satisfies a good large deviations principle on $\mc P(\mc X)$ with rate $I$. It is immediate to check that $\Pi$ preserves the convexity, so $I$ is convex.

Suppose now \refa{4} holds (but not \refa{3}). Consider the map $\Pi' \colon \mc P(\mc Y) \to \mc M_1(\mc X)$ defined by
\begin{equation*}
\Pi'(\nu)(f)= \nu(f) \qquad \forall f\in C_c(\mb X)
\end{equation*}
where we also identified $f$ with its unique continuous extension on $\mc Y$ (namely $f(x)=0$ for $x\in \mc Y \setminus \mc X$). Then $\Pi'$ is continuous, and we conclude again by contraction principle.
\end{proof}
Motivated by the previous remark, hereafter we assume $\mc X$ to be compact, with no loss of generality.

For $\delta>0$, define $\xi_\delta\colon \mc X \to [0,+\infty]$ as
 \begin{equation}
 \label{e:xidelta}
\xi_\delta(x)= \sup \big\{c\,:\: \bar \mu(e^{c \tau} 
                                \un{B_\delta(x)})<+\infty \big\}
\end{equation}
In particular $\xi =\sup_{\delta>0} \xi_\delta$. Let $\hat{\xi}_\delta$ be
the lower semicontinuous envelope of $\xi_\delta$.
 \begin{lemma}
 \label{l:xilsc}
 For all $x\in \mc X$, $\xi(x)=\sup_{\delta>0} \hat{\xi}_\delta(x)$.
 In particular $\xi$ is lower semicontinuous.
\end{lemma}
\begin{proof}
  By the very definition of $\xi_\delta$, if $y \in B_\delta(x)$, then
  $\xi_{2 \delta}(x) \le \xi_{\delta}(y)$. Therefore
  \begin{equation*}
\xi_\delta(x) \ge \hat \xi_\delta(x):= 
   \sup_{\eps>0} \inf_{y \in B_\eps(x)} \xi_\delta(y) 
\ge \inf_{y \in B_\delta(x)}  \xi_\delta(y) \ge \xi_{2 \delta}(x)
\end{equation*}
The lemma follows taking the supremum in $\delta>0$.
\end{proof}

Let $LSC(\mc X)$ be the set of lower semicontinuous functions $f:\mc X
\to ]-\infty,+\infty]$. If $f \in LSC(\mc X)$ then $f$ is bounded from
below.
\begin{lemma}
  \label{l:unifint}
Recall \eqref{e:xidelta}. For all $M<+\infty$ and $\eps,\,\delta>0$ (hereafter $a\wedge b:= \min(a,b)$)
\begin{equation}
\label{e:xiup}
 \bar \mu (e^{(\xi_{\delta} \wedge M -\eps) \tau}) < +\infty .
\end{equation}
On the other hand, if $f \in LSC(\mc X)$ is such that
\begin{equation}
\label{e:xidown}
\bar \mu (e^{\tau f}) <+\infty,
\end{equation}
then $f(x) \le \xi(x)$ for \emph{all} $x \in \mc X$. In particular
\begin{equation}
\label{e:xi2}
\xi(x) = \sup \{f(x),\,f\in LSC(\mc X)\,:\: \bar \mu (e^{\tau f}) <+\infty\} .
\end{equation}
\end{lemma}
\begin{proof}
  Fix $M,\,\eps,\,\delta>0$ and let
  $\{B_{\delta/2}(y_1),\ldots,B_{\delta/2}(y_n)\}$ be a finite
  covering of the compact space $\mc X$ with balls of radius
  $\delta/2$. Since $\xi_{\delta}(x) \le \xi_{\delta/2}(y_i)$ for $x
  \in B_{\delta/2}(y_i)$
  \begin{equation*}
    \begin{split}
& \bar \mu (e^{(\xi_{\delta} \wedge M -\eps) \tau}) 
    \le \sum_{i=1}^n \bar \mu (e^{(\xi_{\delta} \wedge M -\eps) \tau} 
                                \un {B_{\delta/2}(y_i)}) 
    \le \sum_{i=1}^n \bar \mu (e^{(\xi_{\delta/2}(y_i) \wedge M -\eps) \tau} 
                                \un {B_{\delta/2}(y_i)}) .
    \end{split}
  \end{equation*}
  Since $\xi_{\delta/2}(y_i) \wedge M -\eps< \xi_{\delta/2}(y_i)$,
  each term in the summation in the last line of the above formula is
  finite by the very definition of $\xi_{\delta/2}(y_i)$. Thus \eqref{e:xiup} holds.

  Let now $f\in LSC(\mc X)$, and suppose that for some $x \in \mc X$
  and $\eps>0$, $f(x) \ge \xi(x)+2\eps$. Since $f$ is lower
  semicontinuous, there exists $\delta>0$ such that $\inf_{y \in
    B_\delta(x)}f(y)\ge \xi(x)+\eps$. Then
\begin{equation*}
  \bar \mu(e^{\tau f}) \ge \bar \mu(e^{\tau f} \un {B_{\delta}(x)}) 
\ge
  \bar \mu(e^{\tau [\xi(x)+\eps]} \un {B_{\delta}(x)}) 
\ge  
  \bar \mu(e^{\tau [\xi_{\delta}(x)+\eps]} \un {B_{\delta}(x)})=+\infty.
\end{equation*}
Therefore if \eqref{e:xidown} holds, then $f \le \xi$ everywhere.
\end{proof}

\begin{proposition}
\label{p:Ibound}
For each $\nu \in \mc P(\mc X)$
\begin{equation}
 \label{e:Ibound}
 I(\nu) = \sup \big\{\nu(f),\, f\in LSC(\mc X)\,:\: \bar \mu(e^{\tau\,f})
  \le 1 \big\}=:\tilde{I}(\nu).
\end{equation}
In particular Proposition~\ref{p:Ilsc} holds.
\end{proposition}
\begin{proof}
Fix $\nu \in \mc P(\mc X)$, and let $f \colon \mc X \to \bb R$ be Borel measurable, $\nu$-integrable, such that $\bar \mu(e^{\tau\,f})<1$ and $f \le (\hat \xi_\delta \wedge M- \eps)$ for some $M,\,\delta,\,\eps>0$. Since continuous functions are dense in $L_1(\nu+\bar \mu)$, there exists a sequence $(f_n)$ in $LSC(\mc X)$ such that $f_n \to f$ in $L_1(d\nu)$ and (up to passing to a subsequence) also $\bar \mu$-almost everywhere.
 Moreover one can assume $f_n \le \hat \xi_\delta \wedge M-\eps$, since the sequence $f_n \wedge (\hat \xi_\delta
  \wedge M-\eps)$ is in $LSC(\mc X)$ and enjoys the aforementioned properties as well. Dominated convergence and \eqref{e:xiup} imply $\lim_n \bar \mu (e^{\tau f_n}) = \bar \mu (e^{\tau f})<1$.
  Therefore $\bar \mu(e^{\tau f_n}) \le 1$ for $n$ large enough. Thus
\begin{equation}
 \label{e:Ibound2}
 \begin{split}
 & \tilde I(\nu)
\ge  
   \sup_{M,\,\delta,\,\eps>0} \sup \big\{ 
                   \nu(f),\, \text{$f$ $\nu$-integrable such that}\quad
\bar \mu(e^{\tau\,f})< 1,\,  f \le \hat \xi_\delta \wedge M-\eps \big\} .
\end{split}
\end{equation}
By \refa{2}, the Borel set $A =\{\xi=+\infty\} \setminus \mathrm{Supp}(\nu_s)$ is 
such that $\bar \mu$ and $\nu_a$ are concentrated on $A$ and $\nu_s$ is concentrated on $A^c$. Fix
$M,\,\delta,\,\eps>0$ and take $\varphi \in C(\mc X)$ such that $\bar \mu
(e^\varphi)\le 1$. In the right hand side of \eqref{e:Ibound2}
consider a $f$ of the form
\begin{equation}
\label{e:f}
f=\big(\frac{\varphi}{\tau}\wedge 
  \hat \xi_\delta \wedge M \big) \un {A}
+ (\hat \xi_\delta \wedge M)\un  {A^c} -\eps.
\end{equation}
Then $\bar \mu(e^{\tau f})=  \bar \mu(e^{\tau f} \un {A})
\le \bar \mu (e^{\varphi -\eps}) \le
e^{-\eps}<1$.

If $\nu_a(1/\tau)=+\infty$, take $\varphi \equiv 1$ in
\eqref{e:f}. Then $f$ is $\nu$-integrable and by monotone convergence $\nu(f) \to +\infty$ as one lets $M\to
+\infty$ and $\delta \downarrow 0$, so that $\tilde{I}(\nu)=+\infty$
by \eqref{e:Ibound2}. Thus $\tilde I(\nu)=I(\nu)=+\infty$ whenever $\nu_a(1/\tau)=+\infty$.

Consider then the case $\nu_a(1/\tau)<+\infty$. Since $\varphi$ is bounded, any $f$ of the form \eqref{e:f} is $\nu$-integrable,  and thus by \eqref{e:Ibound2}
\begin{equation*}
\tilde{I}(\nu) \ge \nu(f) =
\nu_a\big(\frac{\varphi}{\tau}\wedge \hat \xi_\delta \wedge M\big)+
\nu_s(\hat \xi_\delta \wedge M)-\eps.
\end{equation*}
Take the limit $M \to+\infty$, $\delta \downarrow 0$, $\eps \downarrow 0$. Monotone convergence and Lemma~\ref{l:xilsc} then yield
\begin{equation*}
\tilde{I}(\nu) \ge \nu_a\big(\frac{\varphi}{\tau}\big)+
\nu_s(\sup_{\delta>0} \hat \xi_\delta)= \nu_a\big(\frac{\varphi}{\tau}\big)+ \nu_s(\xi)=\nu_a(1/\tau) \bar \nu_a(\varphi)+\nu_s(\xi)
\end{equation*}
where the last equality is a direct consequence of \eqref{e:nubar}. Now optimize over $\varphi$ to get
\begin{equation*}
  \begin{split}
 &   \tilde{I}(\nu) \ge 
 \nu_a(1/\tau) \sup \big\{ \bar \nu_a(\varphi),
  \,\varphi \in C(\mc X)\,:\:\bar \mu(e^\varphi) \le 1  \big\} + \nu_s(\xi)
\\
& \phantom{\tilde{I}(\nu)}
\ge  \nu_a(1/\tau) \sup \big\{ \bar \nu_a(\varphi) -\log \bar \mu(e^\varphi),
  \,\varphi \in C(\mc X)\,:\:\bar \mu(e^\varphi) = 1  \big\} + \nu_s(\xi)
\end{split}
\end{equation*}
Notice that the condition $\bar \mu(e^\varphi) = 1$ can now be dropped in the supremum in the last line above, since for any $c\in \bb R$ the change $\varphi \mapsto \varphi+c$ leaves the quantity $ \bar \nu_a(\varphi) -\log \bar \mu(e^\varphi)$ invariant. Therefore the supremum over $\varphi$ equals the relative entropy as defined in  \eqref{e:relent1}, so that $\tilde I \ge I$.

In order to prove $I(\nu) \ge \tilde I(\nu)$, one only needs to consider the
case $\nu_a(1/\tau)<+\infty$, the inequality being trivial
otherwise. Then for $\varphi \in L_1(d\bar \nu_a)$ such that $\bar
\mu(e^\varphi)\le 1$,
\begin{equation*}
 \begin{split}
&
\nu_a(1/\tau) \mb H(\bar \nu_a|\bar \mu)
\ge
\nu_a(1/\tau) \big[\bar \nu_a(\varphi)-\log \bar\mu (e^{\varphi}) \big]
\ge \nu_a(\varphi/\tau)=\nu_a(f),
 \end{split}
\end{equation*}
where $f:=\varphi/\tau$ and the above conditions on $\varphi$
translates into $f \in L_1(d\nu_a)$ and $\bar \mu(e^{\tau f})\le
1$. Therefore, optimizing over $f\in LSC(\mc X)$ satisfying these two
conditions, and noting that Lemma~\ref{l:unifint} implies $f \le \xi$
for such a $f$
\begin{equation*}
\begin{split}
&   I(\nu)=\nu_a(1/\tau) \mb H(\bar \nu_a|\bar \mu) +\nu_s(\xi)
\\ &  \phantom{
I(\nu)
}
    \ge \sup \{\nu_a(f),\,f\in LSC(\mc X) \cap L_1(d\nu_a)\,:\:
                        \bar \mu (e^{\tau \,f}) \le 1\}+\nu_s(\xi)
\\ &  \phantom{
I(\nu)
}
    = \sup \{\nu_a(f)+\nu_s(\xi),\,f\in LSC(\mc X)\,:\:
                        \bar \mu (e^{\tau \,f}) \le 1\}
\\ &  \phantom{
I(\nu)
}
 \ge  \sup \{\nu_a(f)+\nu_s(f),\,
              f\in LSC(\mc X)\,:\: \bar \mu (e^{\tau \,f}) \le 1\}
= \tilde I (\nu).
\end{split}
\end{equation*}
Now
\eqref{e:Ibound}  states in particular that $I$ is the supremum of a family of linear lower semicontinuous
mappings, thus Proposition~\ref{p:Ilsc} follows.
\end{proof}

\begin{lemma}
 \label{l:xidown}
For $A \subset \mc X$ a Borel set, define
\begin{equation*}
\xi^{A}:= \sup \big\{c\geq 0\,:\: \bar \mu(e^{c \tau} \un{A})<+\infty
 \big\},
\end{equation*}
\begin{equation*}
\underline \xi^{A}:=- \varlimsup_{L \to +\infty} \frac{1}L \log
\bar \mu\big(\{\tau \ge L \} \cap A\big).
\end{equation*}
Then $\underline \xi^{A}=\xi^{A}$.
\end{lemma}
\begin{proof}

For $c>0$
\begin{equation*}
\bar \mu(e^{c\tau} \un A)= \int_{\bb R^+} d\eta\,\bar \mu(\{e^{c\tau}\ge \eta\}\cap A)=
c\,\int_{\bb R^+}dL\, \bar \mu(\{\tau \ge L\}\cap A) \,e^{c\,L}.
\end{equation*}
It is then easy to check that, for $c>\underline \xi^A$, $\bar \mu(e^{c\tau} \un A)=+\infty$, while if $\xi^A>0$ and $0<c<\xi^A$, then $\bar \mu(e^{c\tau} \un A)<+\infty$. It follows $\xi^A=\underline \xi^A$.
\end{proof}

\begin{proposition}
\label{p:jlsce}
Define $J:\mc P(\mc X) \to [0,+\infty]$ as
\begin{equation}
\label{e:J}
J(\nu)=
\begin{cases}
  I(\nu) & \text{if $\nu=\nu_a$},
\\
+\infty & \text{otherwise}.
\end{cases}
\end{equation}
$I$ is the lower semicontinuous envelope of $J$.
\end{proposition}
Notice that in the classical case $\tau \equiv 1$, $J$ coincides with $I$. However, in this general case, $I=J$ iff $\xi\equiv +\infty$.
\begin{proof}[Proof of Proposition~\ref{p:jlsce}]
  Since $J\ge I$ and $I$ is lower semicontinuous, the lower
  semicontinuous envelope of $J$ is greater than $I$. Therefore it is
  enough to show that for each $\nu \in \mc P(\mc X)$ such that
  $I(\nu)<+\infty$, there exists a sequence $\nu^n \to \nu$ such that
  $\varlimsup_n J(\nu^n) \le I(\nu)$.

 Let $\nu=\nu_a+\nu_s$ satisfy $I(\nu)<+\infty$. Since $\mc X$ is compact, for
  each $\delta\in (0,1)$ there exist $n^\delta \in \bb N^+$ and a
  finite Borel partition $(A^\delta_1,\ldots,\,A^\delta_{n^\delta})$
  of $\mc X$ such that each $A^\delta_i$ has diameter bounded by
  $\delta$, has nonempty interior, and satisfies $\nu_s(\partial A^\delta_i)=0$. For
  $\delta>0$ and $M > L \ge 0$, define
\begin{equation*}
  A_i^{\delta,L,M}:=\{ L\le  \tau \le M\}\cap
  A^\delta_i.
\end{equation*}
Fix a $j \in \{1,\ldots,\,n^\delta\}$. We claim that
\begin{equation}
 \label{e:mutau}
\text{if $\nu_s(A^\delta_j)>0$ then  $\forall L \ge 0,\,\exists M^L\ge L$ such that $\bar
\mu(A_j^{\delta,L,M})>0$ for all $M\ge M^L$.
}
\end{equation}
Indeed $\nu_s(\xi)\le I(\nu)<+\infty$, thus $\nu_s$ is concentrated on $\{\xi <+\infty\}$.
Since $\nu_s(\partial A_j^\delta)=0$, there exists a point $x_j^{\delta}$ in the interior of $A_j^\delta$ such that
$\xi(x_j^\delta)<+\infty$. Then, for each $c>\xi(x_j^\delta)$ and $\eps>0$
\begin{equation*}
\lim_{M\to +\infty} \bar \mu\big( e^{c\tau} \un {B_\eps(x_j^\delta)} \un { L\le  \tau \le M} \big)
=\bar \mu(e^{c\tau} \un {B_\eps(x_j^\delta)} \un {\tau \ge L})=+\infty.
\end{equation*}
Hence for $M$ large enough $\{L\le \tau \le M\}$ has positive $\bar \mu$-measure in each
neighbourhood of $x_j^\delta$, including $A_j^\delta$. The claim \eqref{e:mutau} is thus proved.

By \eqref{e:mutau}, for each $\mb L= (L_1,\,L_2,\ldots) \in [0,+\infty[^{\bb N}$ there exists $\mb M^{\mb L}
 \in [0,+\infty[^{\bb N}$, such that the probability measure
\begin{equation}
\label{e:nudeltaL}
  \nu^{\delta,\mb L,\mb M}(dx):= \nu_a(dx) + \sum_{i=1}^{n^\delta} 
\nu_s(A_i^\delta) \frac{\tau(x) 
 \bar \mu(dx |A^{\delta,L_i,M_i}_i)}{\bar \mu(\tau|A^{\delta,L_i,M_i}_i)}
\end{equation}
is well defined whenever $\mb M\ge \mb M^{\mb L}$, provided the terms in the summation are understood to vanish whenever
$\nu_s(A_i^\delta)$ does. It follows straightforwardly from this definition that for each
$\varphi \in C_b(\mc X)$
\begin{equation}
 \label{e:nuconv}
 \begin{split}
 & \lim_{\delta \downarrow 0}  
\sup_{\mb L \in [0,+\infty[^{\bb N},\,\mb M \ge \mb M^{\mb L}}  \big|\nu^{\delta, \mb L,\mb M}(\varphi) 
        - \nu(\varphi)\big|
        \\
 & \quad \le \varlimsup_{\delta \downarrow 0}  
 \sum_{i=1}^{n^\delta} \nu_s(A_i^\delta) \big[\sup_{x \in A_i^\delta} \varphi(x)- \inf_{x \in A_i^\delta} \varphi(x)\big] 
        =0.
\end{split}
\end{equation}

Note that for each $\delta >0$ and $\mb L,\,\mb M \in [0,+\infty[^{\bb N}$ with $\mb M\ge \mb M^L$, $\nu^{\delta,\mb L,\mb M}$ is absolutely continuous 
with respect to $\bar \mu$. By 
the convexity of $I$ proved in Proposition~\ref{p:Ibound}
\begin{equation}
 \label{e:convexbound}
\begin{split}
 J(\nu^{\delta,\mb L,\mb M} )  = & I(\nu^{\delta,\mb L,\mb M}) \le  \nu_a(\mc X)
              I\Big(\frac{1}{\nu_a(\mc X)}\nu_a\Big)
 \\ & \phantom{ =  I(\nu^{\delta,\mb L}) \le}
+ \sum_{i=1}^{n^\delta} 
\nu_s(A_i^\delta)
I \Big( \frac{\tau(x) 
 \bar \mu(dx |A_i^{\delta,L_i,M_i})}{\bar \mu(\tau|A_i^{\delta,L_i,M_i})}\Big)
\\ & 
= I(\nu)- \Big[ \nu_s(\xi)- \sum_{i=1}^{n^\delta} 
\nu_s(A_i^\delta)
I \Big( \frac{\tau(x) 
 \bar \mu(dx |A_i^{\delta,L_i,M_i})}{\bar \mu(\tau|A_i^{\delta,L_i,M_i})}\Big)\Big]
\end{split}
\end{equation}
where the corresponding terms above are understood to vanish whenever $\nu_a(\mc X)$ or $\nu_s(A_i^\delta)$ do. By direct computation
\begin{equation*}
\begin{split}
I \Big( \frac{\tau(x)
 \bar \mu(dx |A^{\delta,L_i,M_i}_i)}{\bar \mu(\tau|A^{\delta,L_i,M_i}_i)}\Big)
& 
= -\frac{1}{\bar \mu(\tau|A_i^{\delta,L_i,M_i})} 
\log \bar \mu(A_i^{\delta,L_i,M_i})
\\ &
 \le -\frac{1}{L_i} 
\log \bar \mu(\{L_i\le \tau \le M_i\} \cap A_i^{\delta}).
\end{split}
\end{equation*}
Thus, from  Lemma~\ref{l:xidown}
 \begin{equation*}
\varlimsup_{L_i\to +\infty}   \varlimsup_{M_i \to +\infty}
 \Big( \frac{\tau(x)
 \bar \mu(dx |A^{\delta,L_i,M_i}_i)}{\bar \mu(\tau|A^{\delta,L_i,M_i}_i)}\Big)
   \le \xi^{A^{\delta}_i}.
\end{equation*}
Now, since $\xi \ge \xi^{A^{\delta}_i}$ on ${A^{\delta}_i}$
\begin{equation*}
\varlimsup_{\mb L\to +\infty}   \varlimsup_{\mb M \to +\infty} \sum_{i=1}^{n^\delta} 
\nu_s(A_i^\delta)
I \Big( \frac{\tau(x)
 \bar \mu(dx |A_i^{\delta,L_{i,k}^\delta})}{\bar \mu(\tau|A_i^{\delta,L_{i,k}^\delta})}\Big)\Big]
\le \sum_{i=1}^{n^\delta} \nu_s(A_i^\delta)  \xi^{A^{\delta}_i}
\le \nu_s(\xi).
\end{equation*}

Together with \eqref{e:convexbound} this implies
\begin{equation*}
\sup_{\delta>0} \varlimsup_{\mb L\to +\infty}   \varlimsup_{\mb M \to +\infty} 
J(\nu^{\delta,\mb L,\mb M} )  \le I(\nu).
\end{equation*}
Combining this with \eqref{e:nuconv}, by a standard diagonal argument, there exists a sequence $\nu^n= \nu^{\delta^n,\mb L^n,\mb M^n}$ converging to $\nu$ such that $\varlimsup_n I(\nu^n)\le I(\nu)$.
\end{proof}

\section{Large deviations of the empirical measure}
\label{s:3}
The following identity follows immediately from \eqref{e:pi}, and will come handy in this section.
\begin{equation}
\label{e:pi2}
\pi_t=  \frac{1}{t} \sum_{i=1}^{\mc N_t}
\tau(x_i) \delta_{x_i} + \frac{t -S_{\mc N_t}}t \delta_{x_{\mc N_t}+1}.
\end{equation}

\begin{lemma}
 \label{l:freee}
 Let $f \colon \mc X \to [-\infty,+\infty]$ be a measurable function
 such that $\bar \mu(e^{\tau\,f})\le 1$.  Then
\begin{equation*}
  \sup_{t \ge 1} \frac{1}{t} \mb E 
        \exp[t\,\pi_t(f)] <+\infty.
\end{equation*}
\end{lemma}
\begin{proof}
It is enough to prove the result in the case $\bar
  \mu(e^{\tau\,f})=1$. Then define $\bar \mu_f \in \mc P(\mc X)$ as
  \begin{equation*}
\bar \mu_f(dx):= e^{\tau(x)\,f(x)} \bar \mu(dx).
\end{equation*}
Thus
\begin{equation*}
\begin{split}
 \mb E \exp[t\, \pi_t(f)]
& \ = 
 \sum_{n=0}^\infty  \mb E  \exp \big[ \sum_{i=1}^n \tau(x_i)\,f(x_i) 
                + (t-S_n)\,f(x_{n+1}) \big] \un {\mc N_t = n}
\\ 
& \quad =
\sum_{n=0}^\infty  \int_{\mc X^{n+1}} 
            \Big(\prod_{i=1}^n \bar \mu_f(dx_i) \Big) 
                              \bar \mu(dx_{n+1})\,
 \exp[(t-S_{n}) f(x_{n+1})] \un {\mc N_t=n}.
\end{split}
\end{equation*}
Note that $\{\mc N_t=n\} = \{S_n < t \} \cap \{\tau(x_{n+1}) \ge
t-S_n\}$, so that denoting $\zeta_{n,f} \in \mc P([0,+\infty])$ the
law of $S_n=\tau(x_1)+\ldots+\tau(x_n)$ with respect to $\prod_{i=1}^n \bar
\mu_f(dx_i)$
\begin{equation*}
 \mb E \exp[t\, \pi_t(f)] 
   = \sum_{n=0}^\infty \int_{[0,t[}\zeta_{n,f}(ds) 
                 \int_{\{\tau \ge t-s\} } \bar \mu(dx) e^{(t-s) f(x)}.
\end{equation*}
The rightest integral is bounded by $2$, since $e^{(t-s) f(x)} \le 1 +
e^{\tau(x) f(x)}$ on $\{ \tau \ge t-s\}$. Thus
\begin{equation*}
 \frac{1}{t}  \mb E \exp[t\, \pi_t(f)] 
   \le \frac{2}{t} \sum_{n=0}^\infty \zeta_{n,f}([0,t))
 = \frac{2}{t} \sum_{n=0}^\infty  \mb E_f \un {\mc N_t \ge n} 
    = 2\, \mb E_f \frac{\mc N_t}{t},
\end{equation*}
where $\mb E_f$ denotes expectation with respect to $\bar
\mu_f^{\otimes \bb N^+}$. By general renewal theory
\cite[Chapter~V.4]{A}, $\mb E_f \mc N_t/t \to \frac{1}{\bar
  \mu_f(\tau)}<+\infty$ as $t \to +\infty$.
\end{proof}

\begin{proof}[Proof of Theorem~\ref{t:ld1}, upper bound]
  Fix $\mc O$ an open subset of $\mc P(\mc X)$. Then for each $f \in
  LSC(\mc X)$ such that $\bar \mu (e^{\tau f})\le 1$
\begin{equation*}
    \begin{split}
      & \frac{1}{t} \log \mb P_t(\mc O) = \frac{1}{t} \log \mb E e^{-t
        \pi_t(f)} e^{t \pi_t(f)} \un {\pi_t \in \mc O}
      \\
      & \qquad \le \frac{1}{t} \log \big[ e^{-t \inf_{\nu \in \mc O} \nu(f)}
      \mb E e^{t \pi_t(f)}\big] 
       = -\inf_{\nu \in \mc O} \nu(f)+
      \frac{1}{t} \log \mb E e^{t \pi_t(f)}.
     \end{split}
\end{equation*}
By taking the limsup $t \to \infty$, the last term in the above
formula vanishes by Lemma~\ref{l:freee}. Optimizing over $f$
 \begin{equation}
\label{e:laplace}
\varlimsup_t \frac{1}{t} \log \mb P_t(\mc O)
\le -\sup \{ \inf_{\nu \in \mc O} \nu(f), f\in LSC(\mc X)\,:
 \:\bar \mu (e^{\tau f})\le 1\}.
  \end{equation}
  Since \eqref{e:laplace} holds true for each open set $\mc O \subset
  \mc X$, and $\nu \mapsto \nu(f)$ is lower semicontinuous for $f\in
  LSC(\mc X)$, the minimax lemma \cite[Appendix~2, Lemma~3.3]{KL}
  yields
 \begin{equation*}
\varlimsup_t \frac{1}{t} \log \mb P_t(\mc K)
\le -\inf_{\nu \in \mc K} \sup\{ \nu(f), f\in LSC(\mc X)\,:
 \:\bar \mu (e^{\tau f})\le 1\}
  \end{equation*}
  for each compact $\mc K \subset \mc P(\mc X)$. By
  Lemma~\ref{p:Ibound}, the large deviations upper bound then holds
  true on compact sets. But closed sets are compact since $\mc P(\mc
  X)$ is compact.
\end{proof}

The following remark provides a standard approach for proving large
deviations lower bounds, see for instance \cite{M1} and references therein.
\begin{remark}
\label{r:ldgamma}
If for each $\nu \in \mc P(\mc X)$ there exists a sequence
  $(\mb Q_t)$ in $\mc P(\mc P(\mc X))$ such that $\lim_t \mb
  Q_t=\delta_\nu$ narrowly in $\mc P(\mc P(\mc X))$ and
\begin{equation*}
\varlimsup_t
  \frac{1}{t} \mb H(\mb Q_t|\mb P_t) \le J(\nu),
\end{equation*}
then $({\mb P}_t)_{t>0}$ satisfies a large deviations lower bound with rate
given by the lower semicontinuous envelope of $J$.
\end{remark}

For $t>0$ let $\mf F_t$ be the smallest $\sigma$-algebra on $\mc X^{\bb N^+}$ such that the map
 \begin{equation*}
\mc X^{\bb N^+} \ni \mb x \mapsto (x_1,\ldots,\,x_{\mc N_t(\mb x)+1}) \in \cup_{n \in \bb N^+} \mc X^n 
\hookrightarrow \mc X^{\bb N^+} 
 \end{equation*}
is Borel measurable. Note in particular that $\mc N_t 
\colon \mc X^{\bb N^+} \to \bb N$ and
$\pi_t \colon \mc X^{\bb N^+} \to \mc P(\mc X)$ are $\mf F_t$ measurable (with respect to the discrete $\sigma$-algebra of $\bb N$ and the Borel $\sigma$-algebra on $\mc P(\mc X)$ respectively).

\begin{lemma}
\label{l:relent}
Let $\mc Y$ be a Polish space, $F\colon \mc X^{\bb N^+} \to \mc Y$ a $\mf F_t$-Borel measurable map, $(\bar 
\mu_i)_{i\in \bb N^+}$, $(\bar \nu_i)_{i\in \bb N^+}$ be sequences in $\mc P(\mc X)$ and set $\Omega^\mu:=
\prod_{i \in \bb N^+} \bar \mu_i$, $\Omega^\nu:=\prod_{i \in \bb N^+} \bar \nu_i$. Let $\mb P^F,\,\mb Q^F \in \mc 
P(\mc Y)$ be the laws of $F$ under $\Omega^\mu$ and $\Omega^\nu$ respectively. Then
\begin{equation*}
\mb H(\mb Q^F|\mb P^F) \le \sum_{j=1}^\infty  \mb H(\bar \nu_j|\bar \mu_j)\; \Omega^\nu(\mc N_t\ge j-1).
\end{equation*}
In particular, if $\bar \mu_i=\bar \mu$ and $\bar \nu_i=\bar \nu$, then
\begin{equation*}
\mb H(\mb Q^F|\mb P^F) \le \mb H(\bar \nu|\bar \mu)\, \Omega^\nu(\mc N_t+1).
\end{equation*}
\end{lemma}
\begin{proof}
For $r>0$ let (as above) $h(r)= r(\log r -1)+1$, and let $\mf F^{F} \subset \mf F_{t}$ be the $\sigma$-algebra generated by $F
$. Then 
for $\Omega^\mu$-a.e.\ $\mb x$
\begin{equation*}
\frac{d \mb Q^F}{d \mb P^F}(F(\mb x)) 
 =  \frac{d \Omega^\nu \circ F^{-1}}
            {d \Omega^\mu \circ F^{-1}}(F(\mb x))
 = \Omega^\mu \big(\frac{d\Omega^\nu}{d \Omega^\mu} \big| \mf F^F\big)(\mb x).
 \end{equation*}
Therefore changing variables in the integration and using the convexity of $h$
\begin{equation*}
\begin{split}
\mb H(\mb Q^F|\mb P^F) & = \int_{\mc Y} \mb P^F(dy)\, h\big(\frac{d \mb Q^F}{d \mb P^F}(y)\big)
\\ & = 
\int_{\mc X^{\bb N^+}} \!\!\! \Omega^\mu(d\mb x)\,h\big(\Omega^\mu \big(\frac{d\Omega^\nu}{d \Omega^\mu} 
\big| \mf F^F\big)(\mb x)\big)
\le
\int_{\mc X^{\bb N^+}} \!\!\! \Omega^\mu(d\mb x)\,h\big(\Omega^\mu \big(\frac{d\Omega^\nu}{d \Omega^\mu} 
\big| \mf F_t\big)(\mb x)\big).
\end{split}
\end{equation*}
For $n\in \bb N$, and $\mb x$ such that $\mc N_t(\mb x)=n$ one has $\Omega^\mu \big(\frac{d\Omega^\nu}{d \Omega^\mu} \big| \mf F_t\big)(\mb x) 
= \prod_{j=1}^{n+1}\frac{d\nu_j}{d\mu_j}(x_j)$

and thus
\begin{equation*}
\begin{split}
\mb H(\mb Q^F|\mb P^F) & \le
\sum_{n \in \bb N} \int_{\mc X^{n+1}} \prod_{i=1}^{n+1}\mu_i(dx_i)\,h\big(\prod_{j=1}^{n+1}\frac{d\nu_j}{d\mu_j}
(x_j) \big)\, \un {\mc N_t(\mb x)=n}
\\  & = 
\sum_{n \in \bb N} \int_{\mc X^{n+1}} \prod_{i=1}^{n+1}\nu_i(dx_i) \,
               \log \big(\prod_{j=1}^{n+1}\frac{d\nu_i}{d\mu_j}(x_j) \big)\, \un {\mc N_t(\mb x)=n}
\\ & =
\sum_{j \in \bb N^+}  \int_{\mc X^{j}} \prod_{i=1}^{j}\nu_i(dx_i)  \log \frac{d\nu_j}{d\mu_j}(x_j)\, \un {\mc N_t(\mb x)
\ge j-1}.
\end{split}
\end{equation*}
The event  $\{\mc N_t(\mb x) \ge j-1\}$ only depends on $(x_1,\ldots,x_{j-1})$. Therefore the last integral in the above formula splits into a product as
\begin{equation*}
\begin{split}
\mb H(\mb Q^F|\mb P^F) & \le
\sum_{j \in \bb N^+}  \int_{\mc X^{j-1}} \!\! \prod_{i=1}^{j-1 }\nu_i(dx_i)\, \un {\mc N_t(\mb x)\ge j-1}
\int_{\mc X} \nu_j(dx_j)\, \log \frac{d\nu_j}{d\mu_j}(x_j)
\end{split}
\end{equation*}
which is easily rewritten as in the statement.
\end{proof}

\begin{proof}[Proof of Theorem~\ref{t:ld1}, lower bound]
  In view of Proposition~\ref{p:jlsce}, and Remark~\ref{r:ldgamma},
  for each $\nu \in \mc P(\mc X)$ such that $J(\nu)<+\infty$, one
  needs to find a sequence $(\mb Q_t)$ in $\mc P(\mc P(\mc X))$ such
  that $\mb Q_t \to \delta_\nu$ narrowly and $\varlimsup_t \frac{1}{t}
  \mb H(\mb Q_t | \mb P_t) \le J(\nu)$.

  Fix a $\nu \in \mc P(\mc X)$ absolutely continuous with respect to
  $\bar \mu$ and such that $\nu(1/\tau)\in ]0,+\infty[$, and let 
  $\Omega^\nu(d\mb x):= \prod_{i \in \bb N^+} \bar \nu(dx_i)$ as in Lemma~\ref{l:relent}. Set $\mb Q_t:=\Omega^
\nu \circ \pi_t^{-1}$. Since $\nu(1/\tau)<+\infty$, ergodic theorem yields $\lim_t \mb Q_t = \delta_\nu$. On the other hand, since $\pi_t$ is $\mf F_t$ measurable, one may apply Lemma~\ref{l:relent} with $F=\pi_t$ to 
get
 \begin{equation}
\label{e:Hbound}
  \begin{split}
\frac{1}{t} \mb H(\mb Q_t|\mb P_t) & \le \mb H(\bar \nu | \bar \mu) 
     \frac{\Omega^\nu(\mc N_t+1)}{t}.
  \end{split}
\end{equation}
The renewal theorem \cite[Chapter~V.4]{A} implies $\lim_t \Omega^\nu(\mc N_t)/t= \nu(1/\tau)$, which concludes the proof.
\end{proof}

\end{document}